\documentclass[11pt,reqno]{amsart}
\usepackage[T1]{fontenc}
\usepackage{lmodern}
\usepackage[a4paper,margin=30mm]{geometry}
\usepackage{amsmath,amssymb,amsthm,mathtools,mathrsfs,bm}
\usepackage{enumitem}
\usepackage{booktabs,tabularx,array}
\usepackage{aliascnt}
\usepackage{microtype}
\usepackage{xcolor}
\usepackage[backend=biber,style=numeric,sorting=nyt,maxbibnames=99]{biblatex}
\usepackage[unicode,colorlinks=true,linkcolor=blue!55!black,
  citecolor=blue!55!black,urlcolor=blue!55!black]{hyperref}
\usepackage[nameinlink,noabbrev]{cleveref}

\allowdisplaybreaks
\numberwithin{equation}{section}

\newcommand{\R}{\mathbb R}
\newcommand{\C}{\mathbb C}
\newcommand{\HH}{\mathbb H}
\newcommand{\OO}{\mathbb O}
\newcommand{\K}{\mathbb K}
\newcommand{\Id}{I}
\newcommand{\tr}{\operatorname{tr}}

\newcommand{\Sym}{\operatorname{Sym}}
\newcommand{\Ran}{\operatorname{Ran}}
\newcommand{\Ker}{\operatorname{Ker}}
\newcommand{\Span}{\operatorname{span}}
\newcommand{\dist}{\operatorname{dist}}

\newcommand{\sech}{\operatorname{sech}}
\newcommand{\Spin}{\operatorname{Spin}}
\newcommand{\csch}{\operatorname{csch}}
\newcommand{\dd}{\,\mathrm d}
\newcommand{\dv}{\,\mathrm dv}
\newcommand{\ds}{\,\mathrm d\sigma}
\newcommand{\ip}[2]{\left\langle #1,#2\right\rangle}
\newcommand{\norm}[1]{\left\lVert #1\right\rVert}
\newcommand{\cA}{\mathcal A}
\newcommand{\cC}{\mathcal C}
\newcommand{\cH}{\mathcal H}
\newcommand{\cV}{\mathcal V}

\newcommand{\eps}{\varepsilon}
\newcommand{\sph}{\mathbb S}
\newcommand{\RP}{\mathbb{R}\mathrm P}
\newcommand{\CP}{\mathbb{C}\mathrm P}
\newcommand{\HP}{\mathbb{H}\mathrm P}
\newcommand{\OP}{\mathbb{O}\mathrm P}
\newcommand{\RH}{\mathbb{R}\mathrm H}
\newcommand{\CH}{\mathbb{C}\mathrm H}
\newcommand{\QH}{\mathbb{H}\mathrm H}
\newcommand{\OH}{\mathbb{O}\mathrm H}

\theoremstyle{plain}
\newtheorem{theorem}{Theorem}[section]

\newaliascnt{proposition}{theorem}
\newtheorem{proposition}[proposition]{Proposition}
\aliascntresetthe{proposition}

\newaliascnt{lemma}{theorem}
\newtheorem{lemma}[lemma]{Lemma}
\aliascntresetthe{lemma}

\newaliascnt{corollary}{theorem}
\newtheorem{corollary}[corollary]{Corollary}
\aliascntresetthe{corollary}

\newaliascnt{conjecture}{theorem}

\aliascntresetthe{conjecture}

\theoremstyle{definition}
\newaliascnt{definition}{theorem}
\newtheorem{definition}[definition]{Definition}
\aliascntresetthe{definition}

\newaliascnt{assumption}{theorem}
\newtheorem{assumption}[assumption]{Assumption}
\aliascntresetthe{assumption}

\theoremstyle{remark}
\newaliascnt{remark}{theorem}
\newtheorem{remark}[remark]{Remark}
\aliascntresetthe{remark}

\crefname{theorem}{theorem}{theorems}
\crefname{proposition}{proposition}{propositions}
\crefname{lemma}{lemma}{lemmas}
\crefname{corollary}{corollary}{corollaries}
\crefname{definition}{definition}{definitions}
\crefname{assumption}{assumption}{assumptions}
\crefname{remark}{remark}{remarks}
\crefname{section}{section}{sections}

\title[Coupled reciprocal sums in rank-one symmetric spaces]
{Coupled reciprocal-sum inequalities for Neumann eigenvalues\\
 in rank-one symmetric spaces}

\author{Yifeng Meng} \thanks{}
\address{School of Mathematical Sciences,
Fudan University, Shanghai, 200433, China}
\email{yfmeng23@m.fudan.edu.cn}


\begin{document}

\begin{abstract}
Let $M$ be a rank-one symmetric space of compact or noncompact type and real dimension $m=kn$, where $k\in\{1,2,4,8\}$ is the real dimension of the
 underlying division algebra.  Meng and Wang proved sharp isoperimetric
 inequalities for reciprocal sums of the first $p$ or $\ell$ nonzero Neumann
 eigenvalues, with indices strictly below $m$ in the nonconstant-curvature
 cases.  We keep the whole $m$-dimensional transplanted first eigenspace of an
 equimeasurable geodesic ball and treat its mass and stiffness forms as a
 coupled generalized eigenvalue problem.

For every bounded connected smooth domain in a noncompact rank-one symmetric
 space we prove the full sharp estimate
\[
  \sum_{j=1}^{m}\frac{1}{\mu_j(\Omega)}
  \ge \frac{m}{\mu_1(B)},
\]
where $B$ is a geodesic ball with $|B|=|\Omega|$; equality characterizes
 geodesic balls.  The Hopf-direction contribution is encoded by a positive,
 unital, trace-preserving projection on symmetric matrices, and the finite-
 dimensional closure is supplied by a positive-projection Jensen inequality.

In compact type, under the localization hypothesis of Meng--Wang, the same
 full estimate is obtained whenever $\cot^2 R\ge k-1$, where $R$ is the
 radius of the comparison ball, as well as in the isotropic endpoint cases
 and for Hopf-balanced angular second moment.  Consequently the complete
 $m$-term inequality holds for all complex and quaternionic projective
 spaces under the original localization assumption, and for the Cayley
 plane when $R\le \arctan(1/\sqrt7)$.  In the remaining admissible Cayley
 interval we exhibit a finite-dimensional obstruction showing that the
 affine tangent-matrix bounds alone cannot yield the full trace estimate.
 This obstruction concerns the proof mechanism and is not a counterexample
 to the unresolved spectral inequality.
\end{abstract}

\maketitle

\section{Introduction}\label{sec:introduction}

Let $\Omega$ be a bounded connected domain with smooth boundary in a
Riemannian manifold $M$.  We use the nonnegative sign convention for the
Neumann Laplacian,
\begin{equation}\label{eq:neumann-problem}
  \begin{cases}
   -\Delta u=\mu u & \text{in }\Omega,\\
   \partial_\nu u=0 & \text{on }\partial\Omega,
  \end{cases}
\end{equation}
where $\nu$ is the outer unit normal.  Its spectrum, counted with
multiplicity, is
\[
  0=\mu_0(\Omega)<\mu_1(\Omega)\le \mu_2(\Omega)\le\cdots\nearrow\infty.
\]
The classical Szeg\H{o}--Weinberger theorem states that a Euclidean ball is
 the unique maximizer of $\mu_1$ among domains of fixed volume
\cite{Szego1954,Weinberger1956}.  The first-eigenvalue comparison was extended
 to domains in hyperbolic space and in a hemisphere
\cite{AshbaughBenguria1995,Xu1995}; in dimension two, the spherical admissible
 range has subsequently been pushed beyond the hemisphere
\cite{LangfordLaugesen2023}.  Since the first nonzero Neumann eigenvalue of a
 ball in $\R^m$ has multiplicity $m$, Ashbaugh and Benguria proposed the
 sharper harmonic-mean conjecture
\begin{equation}\label{eq:AB-euclidean}
  \sum_{j=1}^{m}\frac1{\mu_j(\Omega)}
  \ge \frac{m}{\mu_1(B)},
\end{equation}
where $B$ is the Euclidean ball satisfying $|B|=|\Omega|$
\cite{AshbaughBenguria1993,Ashbaugh1999,Henrot2006}.  Xia and Wang proved the
 sharp $(m-1)$-term version in Euclidean and hyperbolic space
\cite{XiaWang2023}; Benguria, Brandolini and Chiacchio established the
 corresponding hemispherical estimate \cite{BenguriaBrandoliniChiacchio2020}.
 The full Euclidean conjecture was proved by He, Li and Tang in 2026 by
 retaining the complete transplanted first eigenspace as a coupled matrix
 problem rather than estimating its coordinate functions separately
\cite{HeLiTang2026}.  Shortly afterwards, You and Zhang obtained the full
 $m$-term analogue in hyperbolic space and, under a hemisphere assumption, on
 the sphere \cite{YouZhang2026}.  Related developments include Witten-Laplacian
 analogues and general reciprocal lower bounds
\cite{ChenMao2024,Eddaoudi2025}, as well as Gaussian and radially weighted
 harmonic-mean inequalities
\cite{GaoWang2026,Chiacchio2026,WangSunZhu2026}.

The purpose of this paper is to implement the coupled trial-space mechanism
 in rank-one symmetric spaces whose sectional curvature is not constant.
 The natural geometric setting is formed by the compact projective spaces
\[
  \RP^n,\qquad \CP^n,\qquad \HP^n,\qquad \OP^2,
\]
and their noncompact duals
\[
  \RH^n,\qquad \CH^n,\qquad \QH^n,\qquad \OH^2.
\]
We write $k=\dim_{\R}\K\in\{1,2,4,8\}$ and $m=kn$; in the octonionic case
$n=2$ and $m=16$.  The normalizations are fixed in
\cref{sec:geometry}.  Aithal and Santhanam proved the corresponding
Szeg\H{o}--Weinberger inequality for the first nonzero Neumann eigenvalue
\cite{AithalSanthanam1996}.  Meng and Wang subsequently established sharp
 reciprocal-sum estimates for lower-order eigenvalues in both compact and
 noncompact type \cite{MengWang2024}.  More precisely, in noncompact type
 they proved
\begin{equation}\label{eq:MengWang-noncompact}
  \sum_{j=1}^{p}\frac1{\mu_j(\Omega)}
  \ge \frac{p}{\mu_1(B)},
  \qquad
  p=\begin{cases}
      m-k,&n>1,\\
      m-1,&n=1,
     \end{cases}
\end{equation}
and in compact type, under the localization condition
$\Omega\subset B_{i(M)/4}(q)$, they proved
\begin{equation}\label{eq:MengWang-compact}
  \sum_{j=1}^{\ell}\frac1{\mu_j(\Omega)}
  \ge \frac{\ell}{\mu_1(B)},
  \qquad
  \ell=\begin{cases}
      m-1,&k=1\text{ or }k=m,\\[1mm]
      \left\lfloor\dfrac{m+k}{2}\right\rfloor-1,&1<k<m.
     \end{cases}
\end{equation}
Here and below $i(M)$ denotes the injectivity radius of a compact space.
The estimates are sharp, and equality occurs only for geodesic balls.

The loss in the indices in \eqref{eq:MengWang-noncompact}--
\eqref{eq:MengWang-compact} comes from a coordinatewise upper bound on the
 angular gradient.  The first eigenspace of the comparison ball is,
 however, an irreducible $m$-dimensional object.  We therefore preserve the
 entire trial space and estimate its two Gram forms in Loewner order.  In
 the Euclidean and constant-curvature cases, one trace-free angular
 imbalance matrix appears simultaneously in the mass and stiffness bounds
\cite{HeLiTang2026,YouZhang2026}.  Rank-one projective geometry introduces
 a second, Hopf-direction imbalance.  A central observation of this paper
 is that the second matrix is not independent: it is the image of the first
 under a canonical positive projection on $\Sym(T_oM)$.

\subsection{Main results}

Let $B_R\subset M$ be a geodesic ball satisfying $|B_R|=|\Omega|$, and put
\[
  \lambda_R:=\mu_1(B_R).
\]
Our main noncompact result is the full reciprocal-sum inequality.

\begin{theorem}[Full sum in noncompact type]\label{thm:main-noncompact}
Let $M$ be a noncompact rank-one symmetric space of real dimension $m=kn$.
For every bounded connected domain $\Omega\subset M$ with smooth boundary,
\begin{equation}\label{eq:main-noncompact}
  \sum_{j=1}^{m}\frac1{\mu_j(\Omega)}
  \ge \frac{m}{\mu_1(B_R)}.
\end{equation}
Equality holds if and only if $\Omega$ is a geodesic ball.
\end{theorem}

For $k=1$, \cref{thm:main-noncompact} is the hyperbolic theorem of You and
Zhang \cite{YouZhang2026}.  Its new content is the complete first
 eigenspace in complex hyperbolic, quaternionic hyperbolic, and Cayley
 hyperbolic geometry.  In particular, the index $p$ in
\eqref{eq:MengWang-noncompact} is improved unconditionally to $m$.

For compact type we retain the localization used by Aithal--Santhanam and
Meng--Wang.

\begin{assumption}[Compact localization]\label{ass:compact-localization}
The domain $\Omega$ is contained in a geodesic ball
$B_{i(M)/4}(q)$ of a compact rank-one symmetric space $M$.
\end{assumption}

All compact projective spaces in our normalization have
$i(M)=\pi/2$.  Thus the containing radius is $\pi/8$, every polar radius
 arising from the center chosen below is at most $\pi/4$, and the radius of
 the equimeasurable comparison ball satisfies $R\le\pi/8$.

\begin{theorem}[Full sum in the favorable compact range]
\label{thm:main-compact}
Let $M$ be a compact rank-one symmetric space of real dimension $m=kn$, and
let $\Omega$ satisfy \cref{ass:compact-localization}.  Then
\begin{equation}\label{eq:main-compact}
  \sum_{j=1}^{m}\frac1{\mu_j(\Omega)}
  \ge \frac{m}{\mu_1(B_R)}
\end{equation}
with equality if and only if $\Omega$ is a geodesic ball, provided one of
 the following conditions holds:
\begin{enumerate}[label=\textup{(\roman*)},leftmargin=2.2em]
  \item $k=1$;
  \item $n=1$;
  \item $1<k<m$ and
  \begin{equation}\label{eq:radius-condition}
     \cot^2 R\ge k-1.
  \end{equation}
\end{enumerate}
\end{theorem}

The numerical consequences of \cref{thm:main-compact} are immediate.

\begin{corollary}[Complex, quaternionic, and Cayley cases]
\label{cor:compact-models}
Under \cref{ass:compact-localization}, the full estimate
\eqref{eq:main-compact} holds for every domain in $\CP^n$ and $\HP^n$.
For $M=\OP^2$, it holds whenever
\begin{equation}\label{eq:Rstar}
  R\le R_*:=\arctan\frac1{\sqrt7}.
\end{equation}
In all these cases equality characterizes geodesic balls.
\end{corollary}

Indeed,
\[
  \cot^2\frac{\pi}{8}=3+2\sqrt2,
\]
which is larger than $1$ and $3$, the values $k-1$ for $\CP^n$ and
$\HP^n$.  For the Cayley plane, $k-1=7$ and
\eqref{eq:radius-condition} is precisely \eqref{eq:Rstar}.  Since
$R_*\approx0.36137$ and $\pi/8\approx0.39270$, only a narrow compact Cayley
 interval remains outside the unconditional full-term conclusion.

There is also a geometric condition that removes the radius restriction.
The angular imbalance matrix $Z$ and the canonical projection $\Pi$ are
 introduced in \cref{sec:hopf,sec:coupled-comparison}.

\begin{theorem}[Hopf-balanced compact domains]
\label{thm:hopf-balanced-intro}
Under \cref{ass:compact-localization}, suppose the chosen Weinberger center
 satisfies $\Pi Z=0$.  Then the full estimate \eqref{eq:main-compact}
 holds, with equality only for geodesic balls, without any additional
 restriction on $R$.
\end{theorem}

The compact sign is genuinely different from the noncompact one.  When
$\cot^2R<k-1$, the affine mass and stiffness comparators admit trace-free
 directions for which the desired finite-dimensional trace inequality is
 reversed.  We make this precise in \cref{prop:abstract-obstruction} and
 give an explicit Cayley direction in \cref{rem:cayley-obstruction}.  Thus
 the remaining Cayley problem cannot be closed by the affine tangent
 matrices alone; nonlinear remainder information or a different trial-space
 mechanism would be required.

\subsection{Structure of the proof}

Let $g(r)\omega_v(\xi)$ be the degree-one first Neumann eigenfunctions on
$B_R$, and extend $g$ constantly for $r\ge R$ to a function $G$.  A
Weinberger center $o$ is chosen so that
\[
  P_v(x)=G(r)\ip{v}{\xi},\qquad v\in T_oM,
\]
has zero mean on $\Omega$.  The $m$-dimensional space
$\{P_v:v\in T_oM\}$ has mass and stiffness matrices $M$ and $K$.  The trace
form of the Ritz principle gives
\begin{equation}\label{eq:intro-ritz}
  \sum_{j=1}^{m}\frac1{\mu_j(\Omega)}\ge \tr(K^{-1}M).
\end{equation}
Writing the volume of the radial section of $\Omega$ in direction $\xi$ as
$Y(\xi)$, the volume constraint produces a trace-free matrix
\[
  Z=\int_{\sph^{m-1}}(Y(\xi)-Y_R)(\xi\otimes\xi)\ds(\xi).
\]
A raywise rearrangement and tangent-line argument yields
\[
  M\succeq a\Id+cZ.
\]
The tangent space of a geodesic sphere has a horizontal part and a
$(k-1)$-dimensional Hopf part.  If $Q_\xi$ denotes the projection onto the
Hopf directions, then
\[
  \int(Y-Y_R)Q_\xi\ds=\cC(Z),
  \qquad \cC=k\Pi-\Id,
\]
where $\Pi$ is positive, unital, trace-preserving, self-adjoint, and
idempotent.  The stiffness bounds become
\[
  K\preceq \lambda_Ra\Id-dZ-e\cC(Z)
  \quad\text{in noncompact type},
\]
and
\[
  K\preceq \lambda_Ra\Id-dZ+e\cC(Z)
  \quad\text{in compact type}.
\]
The first sign pattern is closed by a new positive-projection Jensen lemma.
The second is handled by a positive-mass matrix lemma under
$\cot^2R\ge k-1$.  Equality in the stiffness comparison forces almost every
radial section to be $(0,R)$, and hence forces $\Omega$ to be the
comparison ball.

The paper is organized as follows.  Rank-one geometry and the radial ball
 eigenfunction are recalled in \cref{sec:geometry}.  The Hopf projection is
 constructed in \cref{sec:hopf}.  Variational, centering, and raywise tools
 are developed in \cref{sec:tools}.  The coupled mass and stiffness
 comparisons are derived in \cref{sec:coupled-comparison}.  The finite-
 dimensional closure lemmas are proved in \cref{sec:matrix-closure}.
\Cref{sec:noncompact-proof,sec:compact-proof} contain the proofs of the main
 theorems, and \cref{sec:obstruction} analyzes the remaining compact Cayley
 interval.

\section{Rank-one geometry and the ball eigenspace}\label{sec:geometry}

\subsection{Models and normalization}

Let $M$ be one of the projective rank-one symmetric spaces or its
noncompact dual.  We use the standard normalization in which the compact
nonreal spaces have sectional curvatures in $[1,4]$ and the noncompact
nonreal spaces have sectional curvatures in $[-4,-1]$.  The real models
have constant curvature $1$ and $-1$, respectively.  The parameters are
summarized below.

\begin{center}
\begin{tabular}{@{}ccccc@{}}
\toprule
$\K$ & $k=\dim_{\R}\K$ & compact model & noncompact model & $m$ \\
\midrule
$\R$ & $1$ & $\RP^n$ & $\RH^n$ & $n$ \\
$\C$ & $2$ & $\CP^n$ & $\CH^n$ & $2n$ \\
$\HH$ & $4$ & $\HP^n$ & $\QH^n$ & $4n$ \\
$\OO$ & $8$ & $\OP^2$ & $\OH^2$ & $16$ \\
\bottomrule
\end{tabular}
\end{center}

The compact spaces have injectivity radius and diameter $\pi/2$ in this
normalization.  Standard references for rank-one symmetric spaces and the spectral geometry
of the Laplacian include \cite{Helgason2001,Chavel1984,AithalSanthanam1996};
the polar formulas used below also appear in
\cite{CastillonRuffini2019,MengWang2024}.

Fix $o\in M$ and identify $V:=T_oM$ with $\R^m$.  In noncompact type the
exponential map at $o$ is a global diffeomorphism.  In compact type we use
it only inside the injectivity ball.  Write
\[
  x=\exp_o(r\xi),\qquad r>0,\quad \xi\in S(V)\cong\sph^{m-1}.
\]
The Riemannian volume element is
\begin{equation}\label{eq:polar-volume}
  \dv=J(r)\dd r\ds(\xi),
\end{equation}
where
\begin{equation}\label{eq:J-density}
 J(r)=
 \begin{cases}
   \sin^{m-1}r\,\cos^{k-1}r,&\text{compact type},\\
   \sinh^{m-1}r\,\cosh^{k-1}r,&\text{noncompact type}.
 \end{cases}
\end{equation}
The mean curvature of the geodesic sphere $S_r(o)$ is
\begin{equation}\label{eq:H-def}
  H(r)=\frac{J'(r)}{J(r)}=
 \begin{cases}
  (m-1)\cot r-(k-1)\tan r,&\text{compact},\\
  (m-1)\coth r+(k-1)\tanh r,&\text{noncompact}.
 \end{cases}
\end{equation}
Accordingly,
\begin{equation}\label{eq:minus-Hprime}
 -H'(r)=
 \begin{cases}
 \displaystyle \frac{m-k}{\sin^2r}
 +\frac{k-1}{\sin^2r\cos^2r},&\text{compact},\\[3mm]
 \displaystyle \frac{m-k}{\sinh^2r}
 +\frac{k-1}{\sinh^2r\cosh^2r},&\text{noncompact}.
 \end{cases}
\end{equation}

\subsection{Horizontal and Hopf directions}

For $\xi\in S(V)$, let
\[
  P_\xi:=\xi\otimes\xi
\]
be the orthogonal projection onto the radial line.  The orthogonal
complement $\xi^\perp$ splits into
\begin{equation}\label{eq:horizontal-vertical-split}
  \xi^\perp=\cH_\xi\oplus\cV_\xi,
  \qquad \dim\cH_\xi=m-k,\quad \dim\cV_\xi=k-1.
\end{equation}
The subspace $\cV_\xi$ consists of the Hopf, or vertical, directions.  Let
$Q_\xi$ denote its orthogonal projection.  Thus
$\Id-P_\xi-Q_\xi$ is the projection onto $\cH_\xi$.

For $v\in V$, put $\omega_v(\xi)=\ip{v}{\xi}$.  Jacobi-field comparison is
an identity in rank one and gives the following bilinear angular energy:
\begin{equation}\label{eq:angular-bilinear}
 \ip{\nabla_{S_r}\omega_v}{\nabla_{S_r}\omega_w}
 =\ip{\cA_r(\xi)v}{w},
\end{equation}
where
\begin{equation}\label{eq:A-compact}
 \cA_r(\xi)
 =\frac{\Id-P_\xi-Q_\xi}{\sin^2r}
  +\frac{Q_\xi}{\sin^2r\cos^2r}
 =\csc^2r(\Id-P_\xi)+\sec^2r\,Q_\xi
\end{equation}
in compact type, and
\begin{equation}\label{eq:A-noncompact}
 \cA_r(\xi)
 =\frac{\Id-P_\xi-Q_\xi}{\sinh^2r}
  +\frac{Q_\xi}{\sinh^2r\cosh^2r}
 =\csch^2r(\Id-P_\xi)-\sech^2r\,Q_\xi
\end{equation}
in noncompact type.  Taking the trace in
\eqref{eq:angular-bilinear} recovers
\begin{equation}\label{eq:trace-angular}
  \sum_{i=1}^m|\nabla_{S_r}\omega_i|^2=-H'(r).
\end{equation}

\subsection{The first nonzero eigenspace of a geodesic ball}

Let $B_R=B_R(o)$ be a geodesic ball.  In compact type we assume
$R\le\pi/4$.  The first nonzero Neumann eigenvalue has multiplicity $m$,
 and its eigenspace is the degree-one angular sector
\begin{equation}\label{eq:ball-eigenspace}
  \{g(r)\omega_v(\xi):v\in V\}.
\end{equation}
The radial factor can be normalized by $g'(0)=1$ and solves
\begin{equation}\label{eq:radial-ode}
 \begin{cases}
  g''+Hg'+(\lambda_R+H')g=0,&0<r<R,\\
  g(0)=0,\qquad g'(R)=0.
 \end{cases}
\end{equation}
Moreover,
\begin{equation}\label{eq:g-positive}
  g(r)>0\quad(0<r\le R),
  \qquad g'(r)>0\quad(0<r<R).
\end{equation}
These facts follow from separation of variables and the Sturm--Liouville
 description of the degree-one sector; see
\cite{AithalSanthanam1996,MengWang2024}.  Equivalently, $\lambda_R$ is the
 first eigenvalue of
\begin{equation}\label{eq:radial-Rayleigh}
  \frac{\displaystyle\int_0^R
  \bigl(\phi'(r)^2-H'(r)\phi(r)^2\bigr)J(r)\dd r}
  {\displaystyle\int_0^R\phi(r)^2J(r)\dd r},
  \qquad \phi(0)=0.
\end{equation}

Extend $g$ constantly beyond $R$:
\begin{equation}\label{eq:G-extension}
 G(r)=
 \begin{cases}
  g(r),&0\le r\le R,\\
  g(R),&r\ge R.
 \end{cases}
\end{equation}
In compact type $G$ is only used on $[0,\pi/4]$.  We record the radial
monotonicity that drives all rearrangements.

\begin{lemma}[Radial monotonicity]\label{lem:radial-monotonicity}
The function $G$ is nondecreasing and $G'(r)=0$ for $r>R$.  In addition:
\begin{enumerate}[label=\textup{(\roman*)},leftmargin=2.2em]
 \item in noncompact type, $G(r)/\sinh r$ is nonincreasing on $(0,\infty)$;
 \item in compact type, $G(r)/(\sin r\cos r)$ is nonincreasing on
 $(0,\pi/4)$.
\end{enumerate}
Consequently, the two angular weights in
\eqref{eq:A-compact} or \eqref{eq:A-noncompact}, after multiplication by
$G^2$, are nonincreasing on the relevant interval.
\end{lemma}

\begin{proof}
The monotonicity of $G$ follows from \eqref{eq:g-positive} and the constant
extension.  We include the quotient arguments, following the calculation
in \cite{MengWang2024}.

In noncompact type, set
\[
  s(r)=g'(r)-\coth r\,g(r).
\]
The regular expansion of the degree-one solution at the pole gives
$s(0+)=0$, while $s(R)=-\coth R\,g(R)<0$.  If $s$ were positive somewhere,
 it would have a positive interior maximum at a point $r_0$.  At that point
$s'(r_0)=0$, and combining this identity with \eqref{eq:radial-ode} gives
\begin{align*}
 0={}&\bigl(m\coth r_0+(k-1)\tanh r_0\bigr)g'(r_0)\\
 &+\left(\lambda_R-m\csch^2r_0+(k-1)\sech^2r_0\right)g(r_0).
\end{align*}
Since $s(r_0)>0$, we have $g'(r_0)>\coth r_0g(r_0)$.  Substitution in the
 preceding equality would imply
\[
 m\csch^2r_0-(k-1)\sech^2r_0-\lambda_R
 >m\coth^2r_0+(k-1),
\]
which is impossible.  Hence $s\le0$ and $g/\sinh r$ is nonincreasing on
$(0,R)$.  The constant extension preserves the conclusion for $r\ge R$.

In compact type, put
\[
  s(r)=g'(r)-(\cot r-\tan r)g(r).
\]
Again $s(0+)=0$ and $s(R)<0$.  At a hypothetical positive interior maximum
$r_0\in(0,R)$, the equations $s'(r_0)=0$ and
\eqref{eq:radial-ode} yield
\begin{align*}
 0={}&(m\cot r_0-k\tan r_0)g'(r_0)\\
 &+\left(\lambda_R-\frac{m-k}{\sin^2r_0}
 -\frac{k}{\sin^2r_0\cos^2r_0}\right)g(r_0).
\end{align*}
On $(0,\pi/4)$ the first coefficient is positive, and $s(r_0)>0$ implies
$g'(r_0)>(\cot r_0-\tan r_0)g(r_0)$.  A direct simplification then gives
\[
  \lambda_R<2(m+k).
\]
This contradicts the ball estimate $\lambda_R\ge2(m+k)$ for
$R\le\pi/4$ proved in \cite{AithalSanthanam1996}; see also
\cite[Lemma 4.1]{MengWang2024}.  Hence $g/(\sin r\cos r)$ is
nonincreasing.  It is strictly decreasing after $R$ because $G$ is constant
there.

Finally, in noncompact type
\[
 \frac{G^2}{\sinh^2r},\qquad
 \frac{G^2}{\sinh^2r\cosh^2r}
\]
are nonincreasing.  In compact type the second angular weight equals
$(G/(\sin r\cos r))^2$, and the first equals this function times
$\cos^2r$; both are nonincreasing on $(0,\pi/4)$.
\end{proof}

Define the ball integrals
\begin{equation}\label{eq:ball-integrals}
 A_R=\int_0^RG^2J\dd r,
 \qquad Q_R=\int_0^R(G')^2J\dd r,
\end{equation}
and, according to type,
\begin{equation}\label{eq:H-integrals}
 \begin{aligned}
 H_{0,R}^{+}&=\int_0^R\frac{G^2}{\sin^2r}J\dd r,&
 H_{1,R}^{+}&=\int_0^R\frac{G^2}{\sin^2r\cos^2r}J\dd r,\\
 H_{0,R}^{-}&=\int_0^R\frac{G^2}{\sinh^2r}J\dd r,&
 H_{1,R}^{-}&=\int_0^R\frac{G^2}{\sinh^2r\cosh^2r}J\dd r.
 \end{aligned}
\end{equation}
Summing the eigenvalue identities for an orthonormal basis of the
 degree-one eigenspace gives
\begin{equation}\label{eq:ball-energy-identity}
  Q_R+(m-k)H_{0,R}^{\pm}+(k-1)H_{1,R}^{\pm}
  =\lambda_R A_R.
\end{equation}

\section{The angular Hopf projection}\label{sec:hopf}

Let $\Sym(V)$ denote the real symmetric endomorphisms of $V=T_oM$, endowed
with the Hilbert--Schmidt inner product
\[
  \ip{A}{B}_{\mathrm{HS}}=\tr(AB).
\]
The purpose of this section is to express the Hopf projection $Q_\xi$ as a
linear function of the radial projection $P_\xi$.  This turns the two angular
moments occurring in the stiffness matrix into one trace-free matrix and one
fixed linear image of it.

\begin{theorem}[Canonical Hopf projection]\label{thm:hopf-projection}
There is a linear map
\[
  \Pi:\Sym(V)\longrightarrow\Sym(V)
\]
with the following properties.
\begin{enumerate}[label=\textup{(\roman*)},leftmargin=2.2em]
 \item $\Pi$ is positive, unital, trace-preserving, Hilbert--Schmidt
 self-adjoint, and idempotent.
 \item For every $\xi\in S(V)$,
 \begin{equation}\label{eq:Q-via-Pi}
   Q_\xi=k\Pi(P_\xi)-P_\xi.
 \end{equation}
 \item The complex-linear extension of $\Pi$ to $M_m(\C)$ is a positive
 unital map.
\end{enumerate}
Consequently, if
\begin{equation}\label{eq:C-def}
  \cC:=k\Pi-\Id,
\end{equation}
then $\cC$ is positive on real positive semidefinite matrices,
$\cC(P_\xi)=Q_\xi$, $\cC(\Id)=(k-1)\Id$, and
\begin{equation}\label{eq:C-spectral-splitting}
 \cC=(k-1)\Id\quad\text{on }\Ran\Pi,
 \qquad
 \cC=-\Id\quad\text{on }\Ker\Pi.
\end{equation}
\end{theorem}

\begin{proof}
We first treat $\K=\R,\C,\HH$.  Regard $V$ as $\K^n$ with its standard
real inner product.  Let $J_0=\Id$ and let $J_1,\ldots,J_{k-1}$ be right
multiplication by an orthonormal basis of the imaginary part of $\K$.  Thus
all $J_\alpha$ are orthogonal, and the imaginary ones are skew-symmetric.
Define
\begin{equation}\label{eq:Pi-associative}
  \Pi(A)=\frac1k\sum_{\alpha=0}^{k-1}
  J_\alpha A J_\alpha^{\mathsf T}.
\end{equation}
For $\K=\C$ or $\HH$, the conjugations
$A\mapsto J_\alpha A J_\alpha^{\mathsf T}$ form a finite group after signs
are discarded.  Hence their average is the orthogonal projection onto the
common fixed-point space.  This proves idempotence and Hilbert--Schmidt
self-adjointness.  Positivity, unitality, and trace preservation follow
immediately from \eqref{eq:Pi-associative}.  The same formula on complex
matrices is an average of unitary conjugations and is therefore positive and
unital.

For a unit vector $\xi$, the vectors
$J_0\xi,\ldots,J_{k-1}\xi$ form an orthonormal basis of the real $k$-plane
$\K\xi$.  Therefore
\[
  \Pi(P_\xi)=\frac1k\sum_{\alpha=0}^{k-1}P_{J_\alpha\xi}
  =\frac1kP_{\K\xi}.
\]
Since the Hopf space is $\cV_\xi=\K\xi\ominus\R\xi$, we have
$Q_\xi=P_{\K\xi}-P_\xi$, proving \eqref{eq:Q-via-Pi}.

We next treat the Cayley plane.  Use the real spin representation
$V=\R^{16}$ of $\Spin(9)$.  Choose a symmetric Clifford system
$S_0,\ldots,S_8$ satisfying
\begin{equation}\label{eq:clifford-system}
  S_\alpha S_\beta+S_\beta S_\alpha
  =2\delta_{\alpha\beta}\Id.
\end{equation}
The matrices $\Id,S_0,\ldots,S_8$ are mutually orthogonal for the
Hilbert--Schmidt product, each with squared norm $16$.  Define
\begin{equation}\label{eq:Pi-cayley}
  \Pi(A)=\frac1{16}\left((\tr A)\Id+
  \sum_{\alpha=0}^{8}\tr(AS_\alpha)S_\alpha\right).
\end{equation}
Thus $\Pi$ is the Hilbert--Schmidt projection onto
$\Span\{\Id,S_0,\ldots,S_8\}$; in particular it is self-adjoint and
idempotent.  It is unital and trace-preserving because
$\tr S_\alpha=0$.

For $\xi\in S^{15}$, put
$a_\alpha=\ip{S_\alpha\xi}{\xi}$ and
$S(a)=\sum_{\alpha=0}^8a_\alpha S_\alpha$.  The standard Clifford-system
model of the octonionic Hopf fibration gives
\begin{equation}\label{eq:octonionic-hopf-identity}
  \sum_{\alpha=0}^{8}a_\alpha^2=1,
  \qquad S(a)\xi=\xi;
\end{equation}
see \cite{FerusKarcherMunzner1981,OrneaPartonPiccinniVuletescu2013}.
By \eqref{eq:clifford-system}, $S(a)^2=\Id$.  Its $+1$ eigenspace has real
dimension $8$ and is the octonionic line through $\xi$.  Denoting its
orthogonal projection by $P_{\OO\xi}$, we obtain
\[
  P_{\OO\xi}=\frac12(\Id+S(a)),
  \qquad
  \Pi(P_\xi)=\frac1{16}(\Id+S(a))=\frac18P_{\OO\xi}.
\]
It follows that $8\Pi(P_\xi)-P_\xi=Q_\xi$, which is
\eqref{eq:Q-via-Pi} for $k=8$.

The last formula also proves positivity of $\Pi$ on real symmetric
matrices, because every positive semidefinite matrix is a nonnegative
linear combination of real rank-one projections.  To verify the positivity
needed for operator Jensen's inequality, extend \eqref{eq:Pi-cayley}
complex-linearly.  For a complex unit vector $z$, the numbers
$a_\alpha=z^*S_\alpha z$ are real.  Since
$S(a)^2=|a|^2\Id$ and
\[
  |a|^2=z^*S(a)z\le \norm{S(a)}=|a|,
\]
we have $|a|\le1$.  Hence
\[
  \Pi(zz^*)=\frac1{16}(\Id+S(a))\succeq0.
\]
By spectral decomposition, the complex extension is positive on all
Hermitian positive semidefinite matrices.

Finally, \eqref{eq:C-spectral-splitting} follows from $\Pi^2=\Pi$.  Since
$\cC(P_\xi)=Q_\xi\succeq0$, the real rank-one decomposition also shows
that $\cC$ is positive on every real positive semidefinite matrix.
\end{proof}

The following consequences will be used repeatedly.

\begin{corollary}[Algebraic consequences]\label{cor:hopf-algebra}
For every $X\in\Sym(V)$,
\begin{equation}\label{eq:C-trace}
  \tr\cC(X)=(k-1)\tr X
\end{equation}
and
\begin{equation}\label{eq:C-quadratic-bound}
  \tr\bigl(X\cC(X)\bigr)\le(k-1)\tr(X^2).
\end{equation}
Moreover,
\begin{equation}\label{eq:isotropy-identities}
 \int_{S(V)}P_\xi\ds(\xi)=\frac{|\sph^{m-1}|}{m}\Id,
 \qquad
 \int_{S(V)}Q_\xi\ds(\xi)
 =\frac{(k-1)|\sph^{m-1}|}{m}\Id.
\end{equation}
\end{corollary}

\begin{proof}
The trace identity follows from trace preservation of $\Pi$.  Write
$X=X_1+X_0$ with $X_1=\Pi X\in\Ran\Pi$ and
$X_0=(\Id-\Pi)X\in\Ker\Pi$.  The decomposition is Hilbert--Schmidt
orthogonal, and \eqref{eq:C-spectral-splitting} gives
\[
 \tr(X\cC(X))=(k-1)\norm{X_1}_{\mathrm{HS}}^2
 -\norm{X_0}_{\mathrm{HS}}^2
 \le(k-1)\norm{X}_{\mathrm{HS}}^2.
\]
The first isotropy identity is standard.  Apply \eqref{eq:Q-via-Pi}, use
linearity, and use $\Pi(\Id)=\Id$ to obtain the second.
\end{proof}

\begin{remark}\label{rem:constant-curvature-endpoints}
When $k=1$, $Q_\xi=0$ and $\cC=0$.  When $n=1$, the horizontal space
vanishes and $Q_\xi=\Id-P_\xi$.  Thus, for any scalar function $b(\xi)$
with zero spherical mean,
\begin{equation}\label{eq:n-one-W-minus-Z}
  \int_{S(V)}b(\xi)Q_\xi\ds(\xi)
  =-\int_{S(V)}b(\xi)P_\xi\ds(\xi).
\end{equation}
These are precisely the constant-curvature endpoint cases in the compact
argument.
\end{remark}

\section{Variational, centering, and rearrangement tools}\label{sec:tools}

\subsection{The trace Ritz principle}

The coupled argument uses the entire trial space at once.  The appropriate
variational statement is the following trace form of Hersch's principle;
compare \cite{Hersch1961,HileXu1993,HeLiTang2026}.

\begin{lemma}[Trace Ritz principle]\label{lem:trace-ritz}
Let $\Omega$ be connected, and let $P_1,\ldots,P_m\in H^1(\Omega)$ be
linearly independent in $L^2(\Omega)$ and satisfy
$\int_\Omega P_i\dv=0$.  Define
\begin{equation}\label{eq:mass-stiffness-def}
  M_{ij}=\int_\Omega P_iP_j\dv,
  \qquad
  K_{ij}=\int_\Omega\ip{\nabla P_i}{\nabla P_j}\dv.
\end{equation}
Then $M$ and $K$ are positive definite, and
\begin{equation}\label{eq:trace-ritz}
  \sum_{j=1}^{m}\frac1{\mu_j(\Omega)}
  \ge \tr(K^{-1}M).
\end{equation}
\end{lemma}

\begin{proof}
The matrix $M$ is the $L^2$ Gram matrix, hence is positive definite.  If
$a^{\mathsf T}Ka=0$, then the function
$P_a=\sum_i a_iP_i$ has zero weak gradient and is constant on the connected
domain $\Omega$.  Its mean is zero, so it vanishes.  Linear independence
then gives $a=0$, and $K$ is positive definite.

Let $0<\eta_1\le\cdots\le\eta_m$ be the Ritz values of the Neumann
quadratic form restricted to $\Span\{P_1,\ldots,P_m\}$.  Since this trial
space is contained in the orthogonal complement of the constants, the
min--max principle gives $\mu_j(\Omega)\le\eta_j$.  The $\eta_j$ are the
generalized eigenvalues of $Ka=\eta Ma$, so their reciprocals are the
eigenvalues of $K^{-1/2}MK^{-1/2}$.  Therefore
\[
  \sum_{j=1}^{m}\frac1{\mu_j(\Omega)}
  \ge\sum_{j=1}^{m}\frac1{\eta_j}
  =\tr(K^{-1/2}MK^{-1/2})=\tr(K^{-1}M).
\]
\end{proof}

\subsection{A Weinberger center and the trial space}

For $p\in M$ and $x$ outside the cut locus of $p$, write
\[
 r_p(x)=\dist(p,x),\qquad
 \xi_p(x)=\frac{\exp_p^{-1}(x)}{r_p(x)}.
\]
At $x=p$ the vector-valued expressions below are set equal to zero.

\begin{lemma}[Weighted Weinberger center]\label{lem:weinberger-center}
Let $G$ be the function in \eqref{eq:G-extension}.  In noncompact type
there exists $o\in M$ such that
\begin{equation}\label{eq:center-vector}
  \int_\Omega G(r_o(x))\xi_o(x)\dv(x)=0.
\end{equation}
In compact type, under \cref{ass:compact-localization}, such a point can be
chosen in the closed containing ball
$\overline{B_{i(M)/4}(q)}$.  Consequently
$r_o(x)\le\pi/4$ for $x\in\Omega$.
\end{lemma}

\begin{proof}
Set
\[
  \Phi(t)=\int_0^tG(s)\dd s.
\]
In noncompact type consider
$\mathscr F(p)=\int_\Omega\Phi(\dist(p,x))\dv(x)$.  Since $G$ is bounded
and eventually equal to the positive constant $g(R)$, $\Phi(t)$ has linear
growth with positive slope.  The boundedness of $\Omega$ therefore implies
$\mathscr F(p)\to\infty$ as $p$ escapes to infinity.  Thus $\mathscr F$
has a minimizer $o$.  The first variation of distance and dominated
convergence give, for every $v\in T_oM$,
\[
  D\mathscr F(o)[v]
  =-\int_\Omega G(r_o(x))\ip{v}{\xi_o(x)}\dv(x).
\]
The derivative vanishes at a minimizer, proving \eqref{eq:center-vector}.
This energy-minimization form of the centering argument is analogous to
\cite{Laugesen2021}.

In compact type let
$C=\overline{B_{i(M)/4}(q)}$.  This ball lies strictly inside the convexity
radius and is strongly geodesically convex.  The same functional is
well-defined on $C$, because the distance between two points of $C$ is at
most $\pi/4$, and it attains a minimum there.  If a minimizer $p$ lay on
$\partial C$, every minimizing geodesic from $p$ to a point of the open set
$\Omega\subset C$ would initially point into $C$.  If $\nu$ denotes the
outer normal of $\partial C$, this gives
$\ip{\xi_p(x)}{\nu}<0$ on a set of positive measure.  Hence
\[
 D\mathscr F(p)[-\nu]
 =\int_\Omega G(r_p(x))\ip{\xi_p(x)}{\nu}\dv(x)<0,
\]
contradicting minimality.  The minimizer is therefore interior and its
first variation vanishes in every direction.  This proves
\eqref{eq:center-vector}.  The distance estimate follows from the triangle
inequality and $i(M)=\pi/2$.
\end{proof}

Fix such a center $o$ and put $V=T_oM$.  For $v\in V$, define
\begin{equation}\label{eq:trial-Pv}
  P_v(x)=G(r)\omega_v(\xi)=G(r)\ip{v}{\xi},
  \qquad x=\exp_o(r\xi).
\end{equation}
Choose an orthonormal basis $e_1,\ldots,e_m$ of $V$ and write
$P_i=P_{e_i}$.

\begin{lemma}[Admissibility of the coupled trial family]
\label{lem:trial-admissible}
The functions $P_1,\ldots,P_m$ belong to $H^1(\Omega)$, have zero mean, and
are linearly independent in $L^2(\Omega)$.  Their mass and stiffness
matrices in \eqref{eq:mass-stiffness-def} are therefore positive definite,
and \eqref{eq:trace-ritz} applies.
\end{lemma}

\begin{proof}
Equation \eqref{eq:center-vector} gives the zero-mean conditions.  At the
pole, regularity of the degree-one solution gives
$G(r)=g'(0)r+O(r^3)$.  Hence $G(r)/r$ and $G'(r)$ are locally bounded, and
\eqref{eq:angular-bilinear} shows that $P_v$ has locally square-integrable
gradient.  Away from the pole the assertion is immediate, while outside
$B_R(o)$ the radial derivative vanishes.  Thus $P_v\in H^1(\Omega)$.

Suppose $P_v=0$ almost everywhere.  The function has a continuous
representative, so it vanishes on $\Omega$.  For $r>0$, $G(r)>0$; hence
$\ip{v}{\xi}=0$ for every point of a nonempty open subset in polar
coordinates.  If $v\ne0$, the corresponding equatorial set of directions
has empty interior, a contradiction.  Therefore $v=0$, proving linear
independence.  The last statement follows from \cref{lem:trace-ritz}.
\end{proof}

\subsection{A weighted ray rearrangement}

We use a one-dimensional bathtub principle together with the tangent line
of the resulting distribution function.  This is the mechanism that
retains the angular imbalance instead of discarding it.  See
\cite{LiebLoss2001,HeLiTang2026} for related forms.

\begin{lemma}[Weighted ray rearrangement and tangent line]
\label{lem:ray-rearrangement}
Let $I=(0,L)$, where $L\le\infty$, and let $J>0$ be locally integrable.
Set
\[
  V(t)=\int_0^tJ(s)\dd s,
\]
and let $\tau$ be the inverse of $V$ on its range.  If $E\subset I$ is
measurable, put $Y(E)=\int_EJ(r)\dd r$.  For a continuous weight $w$ for
which the following integrals are finite, define
\[
  F(y)=\int_0^{\tau(y)}w(r)J(r)\dd r.
\]
Then:
\begin{enumerate}[label=\textup{(\roman*)},leftmargin=2.2em]
 \item if $w$ is nondecreasing,
 \[
  \int_EwJ\dd r\ge F(Y(E)),
 \]
 and $F$ is convex;
 \item if $w$ is nonincreasing,
 \[
  \int_EwJ\dd r\le F(Y(E)),
 \]
 and $F$ is concave.
\end{enumerate}
Consequently, for $Y_R=V(R)$,
\begin{equation}\label{eq:tangent-increasing}
 F(y)\ge F(Y_R)+w(R)(y-Y_R)
\end{equation}
when $w$ is nondecreasing, while the inequality is reversed when $w$ is
nonincreasing.  If $w$ is strictly decreasing and equality holds in the
rearrangement for $0<Y(E)<V(L)$, then
$E=(0,\tau(Y(E)))$ modulo $J(r)\dd r$-null sets.
\end{lemma}

\begin{proof}
Let $I_y=(0,\tau(y))$.  Since $E$ and $I_{Y(E)}$ have the same
$J(r)\dd r$ measure, the exchanged parts
$E\setminus I_{Y(E)}$ and $I_{Y(E)}\setminus E$ have the same measure.  An
increasing weight is no smaller on the first set than on the second, while
a decreasing weight is no larger.  This proves the two rearrangement
inequalities.  Strict monotonicity gives the equality statement by the same
exchange argument.

Since $F'(y)=w(\tau(y))$, the derivative is nondecreasing or nonincreasing
with $y$.  This proves convexity or concavity and the corresponding tangent
inequalities.
\end{proof}

\section{Coupled mass and stiffness comparisons}\label{sec:coupled-comparison}

Let $o$ be the center from \cref{lem:weinberger-center}.  In noncompact type
set $L=\infty$; in compact type set $L=\pi/4$.  For almost every
$\xi\in S(V)$ define the radial section
\begin{equation}\label{eq:E-xi}
 E_\xi=\{r\in(0,L):\exp_o(r\xi)\in\Omega\},
 \qquad
 Y(\xi)=\int_{E_\xi}J(r)\dd r.
\end{equation}
Put
\begin{equation}\label{eq:YR}
  Y_R=\int_0^RJ(r)\dd r.
\end{equation}
By polar coordinates and $|\Omega|=|B_R|$,
\begin{equation}\label{eq:volume-balance-rays}
  \int_{S(V)}(Y(\xi)-Y_R)\ds(\xi)=0.
\end{equation}

\subsection{The angular imbalance matrices}

Define
\begin{equation}\label{eq:Z-def}
  Z=\int_{S(V)}(Y(\xi)-Y_R)P_\xi\ds(\xi)
\end{equation}
and
\begin{equation}\label{eq:W-def}
  W=\int_{S(V)}(Y(\xi)-Y_R)Q_\xi\ds(\xi).
\end{equation}
The volume balance gives
\begin{equation}\label{eq:Z-trace-free}
  \tr Z=0,
\end{equation}
and \cref{thm:hopf-projection} gives the decisive coupling relation
\begin{equation}\label{eq:W-CZ}
  W=\cC(Z)=k\Pi Z-Z.
\end{equation}
Thus the Hopf imbalance is a fixed linear image of the ordinary angular
second moment.

The nonnegativity of the ray volumes also gives a two-sided order estimate.
Let
\begin{equation}\label{eq:b-def}
  b=\frac{|\sph^{m-1}|Y_R}{m}.
\end{equation}
Then
\begin{equation}\label{eq:Z-order-bound}
  -b\Id\preceq Z\preceq(m-1)b\Id.
\end{equation}
Indeed,
\[
 Z+b\Id=\int_{S(V)}Y(\xi)P_\xi\ds(\xi)\succeq0.
\]
Moreover, since $P_\xi\preceq\Id$,
\[
 Z+b\Id\preceq\left(\int_{S(V)}Y(\xi)\ds(\xi)\right)\Id
 =|\sph^{m-1}|Y_R\Id,
\]
which proves the upper bound.

Set
\begin{equation}\label{eq:a-c-rho-X}
 a=\frac{|\sph^{m-1}|}{m}A_R,
 \qquad c=G(R)^2,
 \qquad \rho=\frac{cY_R}{A_R},
 \qquad X=\frac caZ.
\end{equation}
Then
\begin{equation}\label{eq:X-order-bound}
  -\rho\Id\preceq X\preceq(m-1)\rho\Id,
  \qquad \tr X=0.
\end{equation}

\subsection{Exact matrix formulas and the mass lower bound}

In the basis $e_1,\ldots,e_m$ of $V$, the exact mass matrix is
\begin{equation}\label{eq:M-exact}
  M=\int_{S(V)}\left(\int_{E_\xi}G(r)^2J(r)\dd r\right)
  P_\xi\ds(\xi).
\end{equation}
Since $G^2$ is nondecreasing, \cref{lem:ray-rearrangement} and the tangent
inequality at $Y_R$ give, on each ray,
\begin{equation}\label{eq:mass-ray-bound}
  \int_{E_\xi}G^2J\dd r
  \ge A_R+c(Y(\xi)-Y_R).
\end{equation}
Integration and the first identity in \eqref{eq:isotropy-identities} yield
\begin{equation}\label{eq:M-lower}
  M\succeq a\Id+cZ=a(\Id+X).
\end{equation}
We emphasize that $\Id+X$ need not be positive; the independent fact
$M\succ0$ will be retained in the compact closure argument.

By the radial-angular orthogonality and
\eqref{eq:angular-bilinear}, the exact stiffness matrix is
\begin{equation}\label{eq:K-exact}
 K=\int_{S(V)}\int_{E_\xi}
 \left(G'(r)^2P_\xi+G(r)^2\cA_r(\xi)\right)
 J(r)\dd r\ds(\xi).
\end{equation}

\subsection{The stiffness upper bounds}

\begin{proposition}[Coupled stiffness comparison]
\label{prop:stiffness-comparison}
With $a,c,Z,W$ as above, the following estimates hold.

In noncompact type,
\begin{equation}\label{eq:K-upper-noncompact}
  K\preceq \lambda_Ra\Id-dZ-eW
  =\lambda_Ra\Id-dZ-e\cC(Z),
\end{equation}
where
\begin{equation}\label{eq:d-e-noncompact}
  d=\frac{c}{\sinh^2R},
  \qquad e=\frac{c}{\cosh^2R}.
\end{equation}

In compact type,
\begin{equation}\label{eq:K-upper-compact}
  K\preceq \lambda_Ra\Id-dZ+eW
  =\lambda_Ra\Id-dZ+e\cC(Z),
\end{equation}
where
\begin{equation}\label{eq:d-e-compact}
  d=\frac{c}{\sin^2R},
  \qquad e=\frac{c}{\cos^2R}.
\end{equation}
\end{proposition}

\begin{proof}
For a ray $E_\xi$, set
\[
 q(\xi)=\int_{E_\xi}G'(r)^2J(r)\dd r.
\]
Since $G'=0$ for $r>R$,
\begin{equation}\label{eq:q-ray-upper}
  q(\xi)\le Q_R.
\end{equation}

Consider first noncompact type.  By
\cref{lem:radial-monotonicity}, both weights
\[
  \frac{G^2}{\sinh^2r}
  \quad\text{and}\quad
  \frac{G^2}{\sinh^2r\cosh^2r}
\]
are nonincreasing.  The rearrangement and concave tangent inequalities give
\begin{align}
 \int_{E_\xi}\frac{G^2}{\sinh^2r}J\dd r
 &\le H_{0,R}^-+d(Y(\xi)-Y_R),
 \label{eq:h0-ray-noncompact}\\
 \int_{E_\xi}\frac{G^2}{\sinh^2r\cosh^2r}J\dd r
 &\le H_{1,R}^-+f(Y(\xi)-Y_R),
 \label{eq:h1-ray-noncompact}
\end{align}
where
\[
 f=\frac{c}{\sinh^2R\cosh^2R}=d-e.
\]
Substitute \eqref{eq:q-ray-upper}--\eqref{eq:h1-ray-noncompact} into
\eqref{eq:K-exact}.  The constant part is
\begin{align*}
 &Q_R\int P_\xi\ds
 +H_{0,R}^-\int(\Id-P_\xi-Q_\xi)\ds
 +H_{1,R}^-\int Q_\xi\ds\\
 &\hspace{20mm}=\frac{|\sph^{m-1}|}{m}
 \left(Q_R+(m-k)H_{0,R}^-+(k-1)H_{1,R}^-\right)\Id\\
 &\hspace{20mm}=\lambda_Ra\Id
\end{align*}
by \eqref{eq:ball-energy-identity}.  The imbalance part is
\begin{align*}
 &d\int(Y-Y_R)(\Id-P_\xi-Q_\xi)\ds
 +f\int(Y-Y_R)Q_\xi\ds\\
 &\qquad=-dZ-dW+fW=-dZ-eW,
\end{align*}
where \eqref{eq:volume-balance-rays} was used.  This proves
\eqref{eq:K-upper-noncompact}.

In compact type the same argument uses the decreasing weights
\[
  \frac{G^2}{\sin^2r}
  \quad\text{and}\quad
  \frac{G^2}{\sin^2r\cos^2r}.
\]
Their endpoint values are $d$ and
\[
 f=\frac{c}{\sin^2R\cos^2R}=d+e.
\]
The constant part is again $\lambda_Ra\Id$, whereas the imbalance part is
\[
  -dZ-dW+fW=-dZ+eW.
\]
This proves \eqref{eq:K-upper-compact}.
\end{proof}

The equality case of the stiffness comparison contains the geometric
rigidity needed later.

\begin{lemma}[Rigidity of the stiffness comparison]
\label{lem:stiffness-rigidity}
Let $T_\pm$ denote the right-hand side of
\eqref{eq:K-upper-noncompact} or \eqref{eq:K-upper-compact}, respectively.
If $K=T_\pm$, then
\begin{equation}\label{eq:ray-ball-ae}
  E_\xi=(0,R)
  \quad\text{modulo }J(r)\dd r\text{-null sets for a.e. }\xi.
\end{equation}
Consequently $|\Omega\triangle B_R(o)|=0$.  If $\Omega$ has smooth
boundary, then $\Omega=B_R(o)$.
\end{lemma}

\begin{proof}
The proof of \cref{prop:stiffness-comparison} writes $T_\pm-K$ as the
integral of three nonnegative scalar deficits multiplying the mutually
orthogonal positive semidefinite projections
$P_\xi$, $\Id-P_\xi-Q_\xi$, and $Q_\xi$.  In particular, the radial part is
\[
  (Q_R-q(\xi))P_\xi,
  \qquad Q_R-q(\xi)\ge0.
\]
If the integral matrix is zero, its trace is zero.  The trace integrand is
nonnegative, so $q(\xi)=Q_R$ for almost every $\xi$.  Since
$G'(r)=g'(r)>0$ on $(0,R)$, this implies
$(0,R)\subset E_\xi$ modulo null sets.  Hence $Y(\xi)\ge Y_R$ almost
everywhere.  The volume balance \eqref{eq:volume-balance-rays} forces
$Y(\xi)=Y_R$, and no radial measure can remain outside $(0,R)$.  This proves
\eqref{eq:ray-ball-ae} and then the polar-coordinate identity gives
$|\Omega\triangle B_R(o)|=0$.

For completeness, suppose first that a point of the open set $\Omega$ lies
outside the open ball $B_R(o)$.  A sufficiently small neighborhood has
positive measure outside the ball, contradicting the null symmetric
difference; the same argument applies to a point of $\Omega$ on
$\partial B_R(o)$.  Hence $\Omega\subset B_R(o)$.  Conversely, a point of
$B_R(o)\setminus\Omega$ either has a neighborhood disjoint from $\Omega$,
or lies on the smooth hypersurface $\partial\Omega$.  In both cases the
complement of $\Omega$ has positive measure in a small ball, again a
contradiction.  Thus $B_R(o)\subset\Omega$.
\end{proof}

\section{Finite-dimensional closure lemmas}\label{sec:matrix-closure}

We write $A\succeq B$ when $A-B$ is positive semidefinite.  We shall use
that inversion reverses the Loewner order on positive definite matrices and
that $\tr(AB)\ge0$ whenever $A,B\succeq0$; see
\cite{Bhatia1997}.

\subsection{The aligned trace-free lemma}

The first lemma is the scalar-aligned closure from the Euclidean argument
of He--Li--Tang.

\begin{lemma}[Aligned trace-free closure]\label{lem:aligned-closure}
Let $a,c,d,\lambda>0$ and $Z\in\Sym_m$ with $\tr Z=0$.  Suppose
\begin{equation}\label{eq:aligned-hypotheses}
 M\succeq0,
 \qquad M\succeq a\Id+cZ,
 \qquad 0<K\preceq\lambda a\Id-dZ.
\end{equation}
Then
\begin{equation}\label{eq:aligned-conclusion}
  \tr(K^{-1}M)\ge\frac m\lambda.
\end{equation}
If equality holds, then $Z=0$, $M=a\Id$, and $K=\lambda a\Id$.
\end{lemma}

\begin{proof}
Put $T=\lambda a\Id-dZ$.  From $0<K\preceq T$ and $M\succeq0$,
\[
  \tr(K^{-1}M)\ge\tr(T^{-1}M)
  \ge\tr\bigl(T^{-1}(a\Id+cZ)\bigr).
\]
Diagonalize $Z$, with eigenvalues $z_1,\ldots,z_m$.  The last trace is
\begin{align*}
 \sum_{i=1}^{m}\frac{a+cz_i}{\lambda a-dz_i}
 &=\frac m\lambda+
 \frac{d+\lambda c}{\lambda}
 \sum_{i=1}^{m}\frac{z_i}{\lambda a-dz_i}.
\end{align*}
The function $z\mapsto z/(\lambda a-dz)$ is strictly convex on the
positivity interval, and $\sum_i z_i=0$.  Jensen's inequality proves
\eqref{eq:aligned-conclusion}.

If equality holds, strict convexity gives $z_i=0$ for every $i$, so
$Z=0$ and $T=\lambda a\Id$.  Equality in the second trace comparison and
$T^{-1}>0$ imply $M=a\Id$.  Since $M$ is then positive definite, equality
in the first comparison implies $K^{-1}=T^{-1}$ and hence
$K=\lambda a\Id$.
\end{proof}

\subsection{A Jensen lemma with a positive projection}

The noncompact comparator contains both $X$ and $\Pi X$.  The next lemma is
the main new finite-dimensional ingredient.

\begin{lemma}[Positive-projection Jensen inequality]
\label{lem:positive-projection-jensen}
Let $\Pi:\Sym_m\to\Sym_m$ be positive, unital, trace-preserving,
Hilbert--Schmidt self-adjoint, and idempotent, and assume its
complex-linear extension is positive.  Let $X\in\Sym_m$ satisfy
$\tr X=0$.  If $u,v>0$ and
\begin{equation}\label{eq:projection-T}
  T=\Id-uX-v\Pi X\succ0,
\end{equation}
then
\begin{equation}\label{eq:positive-projection-jensen}
  \tr\bigl(T^{-1}(\Id+X)\bigr)\ge m.
\end{equation}
Equality holds if and only if $X=0$.
\end{lemma}

\begin{proof}
Set
\[
  Y=uX+v\Pi X,
  \qquad T=\Id-Y.
\]
Since $\Pi^2=\Pi$,
\[
  \Pi Y=(u+v)\Pi X,
  \qquad
  X=\frac1uY-\frac{v}{u(u+v)}\Pi Y.
\]
Define
\[
 H=\tr\bigl((\Id-Y)^{-1}Y\bigr),
 \qquad
 D=\tr\bigl((\Id-Y)^{-1}\Pi Y\bigr).
\]
Trace preservation and $\tr X=0$ give $\tr Y=0$.  If
$y_1,\ldots,y_m<1$ are the eigenvalues of $Y$, then strict convexity of
$t\mapsto t/(1-t)$ gives
\begin{equation}\label{eq:H-nonnegative}
  H=\sum_{i=1}^{m}\frac{y_i}{1-y_i}\ge0,
  \qquad H=0\Longleftrightarrow Y=0.
\end{equation}

We claim that $D\le H$.  On the open convex set $\Id-A\succ0$, let
\[
  \Phi(A)=-\log\det(\Id-A).
\]
Convexity and the derivative
$D\Phi(Y)[B]=\tr((\Id-Y)^{-1}B)$ give
\begin{equation}\label{eq:Phi-gradient}
 \tr\bigl((\Id-Y)^{-1}(Y-\Pi Y)\bigr)
 \ge \Phi(Y)-\Phi(\Pi Y).
\end{equation}
Since $\Pi$ is positive and unital,
$\Id-\Pi Y=\Pi(\Id-Y)\succ0$.  The function $t\mapsto-\log t$ is
operator convex.  The Davis--Choi--Jensen inequality
\cite{HansenPedersen2003} therefore yields
\[
 -\log(\Id-\Pi Y)
 \preceq\Pi\bigl(-\log(\Id-Y)\bigr).
\]
Taking traces and using trace preservation gives
$\Phi(\Pi Y)\le\Phi(Y)$.  The left-hand side of
\eqref{eq:Phi-gradient} is $H-D$, proving the claim.

Finally,
\begin{align*}
 \tr\bigl(T^{-1}(\Id+X)\bigr)-m
 &=\tr\bigl((\Id-Y)^{-1}(X+Y)\bigr)\\
 &=\left(1+\frac1u\right)H-
   \frac{v}{u(u+v)}D\\
 &\ge\left(1+\frac1{u+v}\right)H\ge0.
\end{align*}
This proves \eqref{eq:positive-projection-jensen}.  Equality forces $H=0$
and hence $Y=0$.  The linear operator $u\Id+v\Pi$ acts by $u+v$ on
$\Ran\Pi$ and by $u$ on $\Ker\Pi$, so it is invertible; consequently
$X=0$.  The converse is immediate.
\end{proof}

\subsection{A positive-mass compact closure}

The Hopf term has the opposite sign in compact type.  The lower comparator
$\Id+X$ can then be indefinite, and the independent positivity of the true
mass matrix becomes essential.

\begin{lemma}[Positive-mass closure]\label{lem:positive-mass-closure}
Let $X\in\Sym_m$ satisfy $\tr X=0$ and put $A=\Id+X$.  Suppose
$\widehat M,T\in\Sym_m$ satisfy
\begin{equation}\label{eq:positive-mass-hypotheses}
 \widehat M\succeq0,
 \qquad \widehat M\succeq A,
 \qquad 0<T\preceq2\Id,
 \qquad \tr(TA)\le m.
\end{equation}
Then
\begin{equation}\label{eq:positive-mass-conclusion}
  \tr(T^{-1}\widehat M)\ge m.
\end{equation}
\end{lemma}

\begin{proof}
Set
\[
 N=T^{-1/2}\widehat M T^{-1/2},
 \qquad B=T^{-1/2}AT^{-1/2}.
\]
Then $N\succeq0$ and $N\succeq B$.  Because $0<T\preceq2\Id$, the matrix
$P=2T-T^2$ satisfies $0\preceq P\preceq\Id$.  Therefore
\[
 \tr N\ge\tr(PN)\ge\tr(PB).
\]
The second inequality follows from $P\succeq0$ and $N-B\succeq0$.  Since
$P$ commutes with $T$, cyclicity of the trace gives
\[
 \tr(PB)=\tr\bigl((2\Id-T)A\bigr)
 =2\tr A-\tr(TA)\ge2m-m=m.
\]
Since $\tr N=\tr(T^{-1}\widehat M)$, the proof is complete.
\end{proof}

\section{Proof in noncompact type}\label{sec:noncompact-proof}

\begin{proof}[Proof of \cref{thm:main-noncompact}]
By \cref{lem:trial-admissible,lem:trace-ritz},
\begin{equation}\label{eq:noncompact-ritz-start}
  \sum_{j=1}^{m}\frac1{\mu_j(\Omega)}
  \ge\tr(K^{-1}M).
\end{equation}
Let
\begin{equation}\label{eq:T0-noncompact}
  T_0=\lambda_Ra\Id-dZ-e\cC(Z).
\end{equation}
By \eqref{eq:K-upper-noncompact}, $0<K\preceq T_0$, and hence $T_0\succ0$.
Since inversion reverses Loewner order and $M\succeq0$,
\[
  \tr(K^{-1}M)\ge\tr(T_0^{-1}M).
\]
The mass comparison \eqref{eq:M-lower} and $T_0^{-1}\succ0$ then give
\begin{equation}\label{eq:noncompact-two-order-steps}
  \tr(K^{-1}M)
  \ge\tr\bigl(T_0^{-1}(a\Id+cZ)\bigr).
\end{equation}

Use the normalized matrix $X=cZ/a$ and set
\begin{equation}\label{eq:alpha-beta-noncompact}
  \alpha=\frac{d}{\lambda_Rc}
  =\frac1{\lambda_R\sinh^2R},
  \qquad
  \beta=\frac{e}{\lambda_Rc}
  =\frac1{\lambda_R\cosh^2R}.
\end{equation}
Because $\cC=k\Pi-\Id$,
\begin{align*}
 T_0
 &=\lambda_Ra\bigl(\Id-\alpha X-\beta\cC(X)\bigr)\\
 &=\lambda_Ra\bigl(\Id-(\alpha-\beta)X-k\beta\Pi X\bigr).
\end{align*}
Here $u:=\alpha-\beta>0$ and $v:=k\beta>0$.  Therefore
\cref{lem:positive-projection-jensen} yields
\begin{align*}
 \tr\bigl(T_0^{-1}(a\Id+cZ)\bigr)
 &=\frac1{\lambda_R}
 \tr\left[(\Id-uX-v\Pi X)^{-1}(\Id+X)\right]\\
 &\ge\frac m{\lambda_R}.
\end{align*}
Together with \eqref{eq:noncompact-ritz-start} and
\eqref{eq:noncompact-two-order-steps}, this proves
\eqref{eq:main-noncompact}.

Suppose equality holds.  Then equality must hold at every stage of the
preceding chain.  The equality statement in
\cref{lem:positive-projection-jensen} gives $X=0$, hence $Z=0$ and
$T_0=\lambda_Ra\Id$.  Equality in the mass trace comparison gives
$M=a\Id$.  Since $M\succ0$, equality in the inverse-order comparison gives
$K=T_0$.  Thus equality holds in the stiffness comparison, and
\cref{lem:stiffness-rigidity} implies $\Omega=B_R(o)$.  Conversely, on a
geodesic ball the first nonzero eigenvalue has multiplicity $m$, and the
claimed equality is immediate.
\end{proof}

\section{Proofs in compact type}\label{sec:compact-proof}

\subsection{The constant-curvature endpoints}

\begin{proposition}\label{prop:compact-endpoints}
Under \cref{ass:compact-localization}, the estimate
\eqref{eq:main-compact} holds when $k=1$ or $n=1$, and equality
characterizes geodesic balls.
\end{proposition}

\begin{proof}
If $k=1$, then $Q_\xi=0$ and $\cC=0$, so
\eqref{eq:K-upper-compact} reduces to
\[
  K\preceq\lambda_Ra\Id-dZ.
\]
Together with \eqref{eq:M-lower}, the hypotheses of
\cref{lem:aligned-closure} are satisfied.

If $n=1$, then $Q_\xi=\Id-P_\xi$.  Equations
\eqref{eq:volume-balance-rays} and \eqref{eq:n-one-W-minus-Z} give
$W=-Z$, and hence
\[
  K\preceq\lambda_Ra\Id-(d+e)Z.
\]
Again \cref{lem:aligned-closure} applies.  In either case,
\cref{lem:trace-ritz} gives \eqref{eq:main-compact}.  If equality holds,
the equality part of the matrix lemma yields equality in the stiffness
comparison; \cref{lem:stiffness-rigidity} then gives
$\Omega=B_R(o)$.  The converse is immediate.
\end{proof}

\subsection{The favorable projective range}

\begin{proof}[Proof of \cref{thm:main-compact}]
By \cref{prop:compact-endpoints}, only the case $1<k<m$ remains.  Assume
\begin{equation}\label{eq:compact-radius-assumption-proof}
  \cot^2R\ge k-1.
\end{equation}
Normalize the mass matrix and the stiffness comparator by
\begin{equation}\label{eq:compact-normalization}
 \widehat M=\frac1aM,
 \qquad X=\frac caZ,
 \qquad
 \alpha=\frac{d}{\lambda_Rc}=\frac1{\lambda_R\sin^2R},
 \qquad
 \beta=\frac{e}{\lambda_Rc}=\frac1{\lambda_R\cos^2R}.
\end{equation}
Then
\begin{equation}\label{eq:Mhat-lower}
  \widehat M\succeq0,
  \qquad \widehat M\succeq\Id+X,
\end{equation}
and
\begin{equation}\label{eq:compact-T-def}
  0<K\preceq\lambda_RaT,
  \qquad
  T=\Id-\alpha X+\beta\cC(X)\succ0.
\end{equation}
The radius condition is equivalent to
\begin{equation}\label{eq:alpha-ge-kminus1-beta}
  \alpha\ge(k-1)\beta.
\end{equation}

We verify the remaining hypotheses of
\cref{lem:positive-mass-closure}.  From \eqref{eq:X-order-bound} and
positivity of $\cC$,
\[
  -\alpha X\preceq\alpha\rho\Id,
  \qquad
  \beta\cC(X)\preceq
  \beta(m-1)(k-1)\rho\Id.
\]
Therefore
\begin{equation}\label{eq:T-upper-before-two}
 T\preceq
 \left[1+\rho\bigl(\alpha+(m-1)(k-1)\beta\bigr)\right]\Id.
\end{equation}
The compact angular weights are nonincreasing, so their integrals over
$(0,R)$ are bounded below by their endpoint values times $Y_R$.  The ball
identity \eqref{eq:ball-energy-identity} consequently gives
\begin{align*}
 \lambda_RA_R
 &\ge Y_R\bigl((m-k)d+(k-1)(d+e)\bigr)\\
 &=Y_R\bigl((m-1)d+(k-1)e\bigr).
\end{align*}
After division by $\lambda_RA_R$, this becomes
\begin{equation}\label{eq:rho-bound-compact}
  \rho\bigl((m-1)\alpha+(k-1)\beta\bigr)\le1.
\end{equation}
Moreover, \eqref{eq:alpha-ge-kminus1-beta} implies
\[
 \alpha+(m-1)(k-1)\beta
 \le(m-1)\alpha+(k-1)\beta.
\]
Combining this with \eqref{eq:T-upper-before-two} and
\eqref{eq:rho-bound-compact} gives
\begin{equation}\label{eq:T-upper-two}
  0<T\preceq2\Id.
\end{equation}

Next, \eqref{eq:Z-trace-free}, \eqref{eq:C-trace}, and
\eqref{eq:C-quadratic-bound} imply
\begin{align}
 \tr\bigl(T(\Id+X)\bigr)
 &=m-\alpha\tr(X^2)+\beta\tr\bigl(X\cC(X)\bigr)\notag\\
 &\le m-\bigl(\alpha-(k-1)\beta\bigr)\tr(X^2)\notag\\
 &\le m.
 \label{eq:TA-compact}
\end{align}
Thus \cref{lem:positive-mass-closure}, applied to
\eqref{eq:Mhat-lower}, \eqref{eq:T-upper-two}, and
\eqref{eq:TA-compact}, yields
\[
  \tr(T^{-1}\widehat M)\ge m.
\]
By \eqref{eq:compact-T-def} and the inverse Loewner order,
\[
 \tr(K^{-1}M)
 \ge\frac1{\lambda_R}\tr(T^{-1}\widehat M)
 \ge\frac m{\lambda_R}.
\]
The trace Ritz principle proves \eqref{eq:main-compact}.

If equality holds in the spectral inequality, all terms in the chain
\[
 \sum_{j=1}^{m}\frac1{\mu_j(\Omega)}
 \ge\tr(K^{-1}M)
 \ge\frac1{\lambda_R}\tr(T^{-1}\widehat M)
 \ge\frac m{\lambda_R}
\]
are equal.  Since $M\succ0$, equality in the inverse-order comparison
forces $K=\lambda_RaT$, namely equality in the compact stiffness
comparison.  By \cref{lem:stiffness-rigidity}, $\Omega=B_R(o)$.  The
converse is immediate.
\end{proof}

\begin{proof}[Proof of \cref{cor:compact-models}]
Under \cref{ass:compact-localization}, $R\le\pi/8$.  Since
\[
  \cot^2(\pi/8)=3+2\sqrt2>3,
\]
condition \eqref{eq:radius-condition} is automatic for $k=2$ and $k=4$.
For $k=8$ it is equivalent to $\cot^2R\ge7$, or
$R\le\arctan(1/\sqrt7)$.  Apply \cref{thm:main-compact}.
\end{proof}

\subsection{Hopf-balanced angular moment}

\begin{definition}\label{def:hopf-balanced}
The Weinberger center $o$ is called \emph{Hopf-balanced} if the angular
imbalance matrix in \eqref{eq:Z-def} satisfies
\begin{equation}\label{eq:hopf-balanced}
  \Pi Z=0.
\end{equation}
\end{definition}

\begin{proof}[Proof of \cref{thm:hopf-balanced-intro}]
If $\Pi Z=0$, then $\cC(Z)=-Z$.  The compact stiffness estimate becomes
\[
  K\preceq\lambda_Ra\Id-(d+e)Z.
\]
Together with \eqref{eq:M-lower}, this is exactly the setting of
\cref{lem:aligned-closure}.  The trace Ritz principle proves the full sum.
Equality is rigid by \cref{lem:stiffness-rigidity}.
\end{proof}

\section{The compact affine obstruction}\label{sec:obstruction}

The condition $\cot^2R\ge k-1$ is not merely an artifact of the estimate
$T\preceq2\Id$.  When it fails, the two affine comparators themselves no
longer imply the desired trace inequality.

\begin{proposition}[Finite-dimensional obstruction]
\label{prop:abstract-obstruction}
Assume $1<k<m$ and
\begin{equation}\label{eq:bad-radius}
  \cot^2R<k-1.
\end{equation}
Let
\begin{equation}\label{eq:obstruction-alpha-beta}
  \alpha=\frac1{\lambda_R\sin^2R},
  \qquad
  \beta=\frac1{\lambda_R\cos^2R}.
\end{equation}
Then there exists a nonzero trace-free $X\in\Ran\Pi$ such that
$\Id+X\succ0$ and
\[
  T=\Id-\alpha X+\beta\cC(X)\succ0,
\]
but
\begin{equation}\label{eq:obstruction-trace}
  \tr\bigl(T^{-1}(\Id+X)\bigr)<m.
\end{equation}
Consequently, the abstract implications
\[
  \widehat M\succeq0,
  \quad \widehat M\succeq\Id+X,
  \quad 0<K\preceq T
  \quad\Longrightarrow\quad
  \tr(K^{-1}\widehat M)\ge m
\]
are false in the range \eqref{eq:bad-radius}.
\end{proposition}

\begin{proof}
There is a nonzero trace-free element of $\Ran\Pi$: in the complex and
quaternionic cases with $n\ge2$, one may take a real symmetric
$\K$-linear operator with $\K$-diagonal entries $1,-1,0,\ldots,0$; in the
Cayley case one may take any $S_\alpha$ from
\eqref{eq:clifford-system}.  For $X\in\Ran\Pi$,
$\cC(X)=(k-1)X$, so
\begin{equation}\label{eq:T-I-plus-sX}
  T=\Id+sX,
  \qquad s=(k-1)\beta-\alpha>0.
\end{equation}
The ball lower bound $\lambda_R\ge2(m+k)$ and $R\le\pi/4$ imply
\[
  0<s<\frac{2(k-1)}{\lambda_R}<1.
\]
After multiplying a fixed nonzero trace-free element of $\Ran\Pi$ by a
sufficiently small scalar, both $\Id+X$ and $\Id+sX$ are positive definite.

Let $x_1,\ldots,x_m$ be the eigenvalues of $X$.  Then
$\sum_i x_i=0$ and
\[
 \tr\bigl(T^{-1}(\Id+X)\bigr)
 =\sum_{i=1}^{m}h_s(x_i),
 \qquad h_s(t)=\frac{1+t}{1+st}.
\]
Because $0<s<1$,
\[
  h_s''(t)=-\frac{2s(1-s)}{(1+st)^3}<0
\]
on the relevant interval.  Strict concavity and $X\ne0$ give
\[
  \frac1m\sum_{i=1}^{m}h_s(x_i)
  <h_s\left(\frac1m\sum_i x_i\right)=h_s(0)=1,
\]
which is \eqref{eq:obstruction-trace}.  Taking
$\widehat M=\Id+X$ and $K=T$ proves the final assertion.
\end{proof}

The obstructing directions are compatible with genuine radial domains; the
point is that the exact ray integrals also contain nonlinear remainders,
which the affine comparators omit.

\begin{lemma}[Realization of small imbalance directions]
\label{lem:imbalance-realization}
Let $A\in\Sym(V)$ be trace-free.  For every $R$ strictly below the maximal
polar radius under consideration, there are smooth, centrally symmetric,
star-shaped domains $\Omega_\eps$ with $|\Omega_\eps|=|B_R|$ such that
\begin{equation}\label{eq:Z-realization}
  Z_\eps=\frac{2|\sph^{m-1}|}{m(m+2)}\eps A.
\end{equation}
For sufficiently small $|\eps|$, the domains remain in the same compact
localization ball whenever $R<\pi/8$.
\end{lemma}

\begin{proof}
Define the desired ray volume by
\[
  Y_\eps(\xi)=Y_R+\eps\ip{A\xi}{\xi}.
\]
For small $|\eps|$ this is positive and lies in the range of $V$.  Since
\[
  \int_{S(V)}\ip{A\xi}{\xi}\ds
  =\frac{|\sph^{m-1}|}{m}\tr A=0,
\]
the total volume is unchanged.  Put
$r_\eps(\xi)=V^{-1}(Y_\eps(\xi))$ and take the radial graph
$0<r<r_\eps(\xi)$.  It is smooth for small $\eps$ and centrally symmetric
because $Y_\eps(-\xi)=Y_\eps(\xi)$; in particular $o$ is a Weinberger
center.

The fourth-moment identity on the sphere is
\[
 \int_{S^{m-1}}\xi_i\xi_j\xi_p\xi_q\ds
 =\frac{|\sph^{m-1}|}{m(m+2)}
 (\delta_{ij}\delta_{pq}+\delta_{ip}\delta_{jq}
 +\delta_{iq}\delta_{jp}).
\]
Using $\tr A=0$ in the definition of $Z$ gives
\eqref{eq:Z-realization}.  The localization statement follows by uniform
convergence $r_\eps\to R$.
\end{proof}

\begin{remark}[Explicit Cayley obstruction]\label{rem:cayley-obstruction}
Let $M=\OP^2$ and
\[
  R_*<R<\frac\pi8,
  \qquad R_*=\arctan\frac1{\sqrt7}.
\]
Take $X=tS_0$, where $S_0$ is one member of the Clifford system and
$0<t<1$.  Then $X\in\Ran\Pi$, $\tr X=0$, and the eigenvalues of $X$ are
$t$ and $-t$, each with multiplicity eight.  Since $\cC(X)=7X$, put
$s=7\beta-\alpha\in(0,1)$.  Formula \eqref{eq:T-I-plus-sX} gives
\[
 \tr\bigl(T^{-1}(\Id+X)\bigr)
 =8\frac{1+t}{1+st}+8\frac{1-t}{1-st}<16.
\]
By \cref{lem:imbalance-realization}, the same matrix direction occurs for
arbitrarily small smooth star-shaped perturbations that remain inside the
original localization range.  This does not disprove the full spectral
inequality: the exact mass and stiffness matrices retain nonlinear raywise
remainders that are absent from the affine model.
\end{remark}

\section{Consequences and remaining question}\label{sec:consequences}

The improvement over the indices in \eqref{eq:MengWang-noncompact} and
\eqref{eq:MengWang-compact} is summarized in
\cref{tab:indices}.

\begin{table}[ht]
\centering
\caption{Indices in the sharp reciprocal-sum inequality.}
\label{tab:indices}
\small
\begin{tabularx}{\textwidth}{@{}>{\raggedright\arraybackslash}p{.28\textwidth}
 >{\raggedright\arraybackslash}X
 >{\raggedright\arraybackslash}X@{}}
\toprule
Geometry & Meng--Wang & Present result \\
\midrule
Noncompact rank one, all $k,n$
 & $p=m-k$ for $n>1$; $p=m-1$ for $n=1$
 & $\boldsymbol{p=m}$ \\
Compact, $k=1$
 & $\ell=m-1$
 & $\boldsymbol{\ell=m}$ \\
$\CP^n$
 & $\ell=n$
 & $\boldsymbol{\ell=2n=m}$ \\
$\HP^n$
 & $\ell=2n+1$
 & $\boldsymbol{\ell=4n=m}$ \\
$\OP^2$, $R\le R_*$
 & $\ell=11$
 & $\boldsymbol{\ell=16=m}$ \\
$\OP^2$, $R_*<R\le\pi/8$
 & $\ell=11$
 & $11$ unconditionally; $16$ if $\Pi Z=0$ \\
\bottomrule
\end{tabularx}
\end{table}

Thus the full first eigenspace is recovered in every noncompact rank-one
symmetric space.  Under the original compact localization, the same is true
for real, complex, and quaternionic projective geometry, and for the Cayley
plane up to the explicit threshold $R_*$.  The only remaining case left by
this method is
\[
  M=\OP^2,
  \qquad R_*<R\le\frac\pi8,
  \qquad \Pi Z\ne0.
\]
\Cref{prop:abstract-obstruction} shows that a proof in this interval must
use information beyond the affine tangent lines.  Natural possibilities
include a quantitative use of the strict concavity remainders on each ray,
a geometric restriction on the realizable mass--stiffness pair, or a trial
space incorporating higher angular structure.

\section*{Acknowledgments}
The author would like to express his sincere gratitude to his supervisor Prof. Bobo Hua for helpful suggestions and encouragement.

The author also thanks Prof. Kui Wang for his interest in this work and valuable feedback.

\section*{Data availability}
No data sets were generated or analyzed in this work.

\printbibliography[title={References}]

\end{document}